\documentclass[a4paper,12pt]{amsart}

\usepackage{mathrsfs}
\usepackage{latexsym,, amsfonts,amssymb}
\usepackage{graphicx,epsfig,subfigure}
\usepackage{psfrag}
\usepackage{enumerate}
\usepackage[latin1]{inputenc}

\newtheorem{teo}{Theorem}[section] \newtheorem{cor}[teo]{Corollary}
\newtheorem{lem}[teo]{Lemma}
\newtheorem{ex}[teo]{Example}
\newtheorem{pro}[teo]{Proposition}

\newtheorem{rem}[teo]{Remark}
\newtheorem*{thmA}{Theorem A}
\newtheorem*{thmB}{Theorem B}

\newcommand{\R}{\mathbb R}
\newcommand{\C}{\mathbb C}
\newcommand{\N}{\mathbb N}

\newcommand{\spec}{\operatorname{{Spec}}}

\newcommand{\codim}{\operatorname{{codim}}}
\newcommand{\dg}{\operatorname{{degree}}}
\newcommand{\vazio}{\varnothing}
\newcommand{\eps}{\varepsilon}

\newcommand{\F}{\mathcal{F}}
\newcommand{\G}{\mathcal{G}}
\newcommand{\el}{\ell}

\newenvironment{dem}
  {\noindent \emph{Proof.}}{$\textrm{\hspace{2mm}} \Box$
  \newline}

\title{On a weak Jelonek's real Jacobian Conjecture in $\R^n$}

\author{Alexandre Fernandes}
\author{Carlos Maquera}
\author{Jean Venato Santos}

\thanks{The first author was supported by CNPq Grant 300685/2008-4, Brazil.}
\thanks{The second author was supported by CNPq and FAPESP Grant 2009/17493-4, Brazil.}
\thanks{The third author was partially supported by CAPES Grant 4441/08-7, Brazil}

\address{Departamento de Matem\'atica,
Universidade Federal do Cear\'a,  13560-970, Fortaleza-CE,
Brazil.} \email{alexandre.fernandes@ufc.br (A. Fernandes)}
\address{Departamento de Matem\'atica,
Instituto de Ci\^encias Matem\'aticas e de
Computa\c{c}\~{a}o, Universidade de S\~ao Paulo - Campus de
S\~{a}o Carlos, Caixa Postal 668, 13560-970, S\~ao Carlos-SP,
Brazil.} \email{cmaquera@icmc.usp.br (C. Maquera)}
\address{Faculdade de Matemática, Universidade Federal de Uberlândia,
38408-100, Uberlândia-MG, Brazil.}
\email{jvenatos@famat.ufu.br (J. Venato Santos)}

\begin{document}
\begin{abstract}
Let $Y:\R^n\to\R^n$ be a polynomial local diffeomorphism and
let $S_Y$ denote the set of not proper points of $Y$. The
Jelonek's real Jacobian Conjecture states that if
$\codim(S_Y)\geq2$, then $Y$ is bijective. We prove a weak
version of such conjecture establishing the sufficiency of a
necessary condition for bijectivity. Furthermore, we
generalize our result on bijectivity to semialgebraic local
diffeomorphisms.
\end{abstract}
\maketitle

\section{Introduction}

In this paper we are interested in finding conditions on a
local diffeomorphism for this to be a global diffeomorphism.
For this, in some cases (for example in the case of
polynomial maps of $\mathbb{R}^n$) it is sufficient the map
to be globally injective. There is an extensive literature,
and in different contexts, on the problem of global
injectivity (see \cite{CGL}, \cite{FGR-n}, \cite{FGR},
\cite{G}, \cite{CGL}, \cite{CJLT}, \cite{GS}, \cite{GM},
\cite{NX}, \cite{B-1} and \cite{B-2}, for instance). In
particular, we put emphasis on two type of conditions that
imply global injectivity as it will be explained below.

\vspace{.3cm}
\noindent
 {\bf I. The spectral condition}: Let $Y:\mathbb{R}^2 \to \mathbb{R}^2$ be a differentiable  map and let $\textrm{Spec}(Y)$ be the set of
 eigenvalues of the derivative $DY_p$ when $p$ varies in $\mathbb{R}^n$
 . In \cite{FGR} Fernandes-Gutierrez-Rabanal  showed that:
If $\textrm{Spec}(Y)\cap [0, \varepsilon)=\emptyset,$ for some $\varepsilon >0,$ then $Y$ is injective. This result is an improvement of several results obtained previously by Gutierrez et al (see \cite{G}, \cite{CGL} and \cite{GS}).
As a consequence of this result the authors also prove the global asymptotic stability conjecture for differentiable vector fields of $\mathbb{R}^2$:  If $\textrm{Spec}(Y)\subset \{z\in \mathbb{C}: \  \textrm{Re}(z) < 0\}$, then for all $p$ in $\mathbb{R}^2$, there is a unique positive trajectory starting at $p$  whose  $\omega$-limit set is exactly $\{0\}$.

 \vspace{.3cm}
 \noindent
 {\bf II. The set of points where a polynomial map is not proper}: First of all we introduce some concepts
and results. Let $Y:M\to N$ be a continuous map of locally
compact spaces. We say that the mapping $Y$ is \textit{not
proper at a point} $y\in N$, if there is no neighborhood $U$
of the point $y$ such that the set $Y^{-1}(\overline{U})$ is
compact. The set $S_Y$ of points at which the map $Y$ is not
proper indicates how the map $Y$ differs from a proper map.
In particular, $Y$ is proper if and only if this set is empty.
Moreover, if $Y(M)$ is open, then $S_Y$ contains the border
of $Y(M)$. The set $S_Y$ is the minimal set $S$ with the
property that the mapping $Y:M\setminus Y^{-1}(S)\to
N\setminus S$ is proper.

In \cite{J}, Jelonek stated the following conjecture for
polynomial maps of $\mathbb{R}^n$:

\vspace{.3cm}
\noindent
\textbf{Jelonek's Real Jacobian Conjecture:}
\textit{
Let
$Y:\R^n\to\R^n$ be a polynomial local diffeomorphism. If
$\codim(S_Y)\geq2$ then $Y$ is bijective.
}

Such conjecture is closely connected with the following
famous Keller Jacobian Conjecture:

\vspace{.3cm}
\noindent
\textbf{Jacobian Conjecture:}
\textit{
Let $Y:\C^n\to\C^n$ be a
polynomial map with nonzero Jacobian determinant everywhere,
then $Y$ is an isomorphism.
}

In fact, Jelonek proved in \cite{J} that if his Real Jacobian
Conjecture in dimension $2n$ is true then so is the Jacobian
Conjecture in (complex) dimension $n$.

Jelonek proved in \cite{J} that: if $Y:\R^n\to\R^n$ is a real
polynomial mapping with nonzero Jacobian everywhere and
$\codim(S_Y)\geq3$, then $Y$ is a bijection (and consequently
$S_Y=\vazio$). Furthermore, he proved that his conjecture is
true in dimension two. This implies that in the not injective
polynomial mapping $P:\R^2\to\R^2$ with nonzero Jacobian
determinant everywhere given by Pinchuk in \cite{Pi}, the
codimension of $S_P$ is equal 1. Based on this example we
have:

\begin{ex}
\label{ex:Pinchuk} The polynomial local diffeomorphism
$Y:\R^n\to\R^n$, $n\geq3$, defined by
$Y(x_1,\ldots,x_n)=(P(x_1,x_2),x_3,\ldots,x_n)$, is not
injective and $\codim(S_Y)=1$.
\end{ex}

Hence the only interesting case is that of $\codim(S_Y)=2$.
In this direction, recently Gutierrez and Maquera proved in
\cite{GM} the following weaker version: if $Y:\R^3\to\R^3$ is
a polynomial map such $\spec(Y)\cap[0,\eps)=\vazio$, for some
$\eps>0$, and $\codim(S_Y)\geq2$, then $Y$ is bijective.

\vspace{.3cm}
\noindent
 {\bf Our results}: Let  $Y=(f_1,\ldots,f_n):\R^n\to\R^n$ be a $C^2$ map such
 that $DY_p$ is non-singular for all $p\in\R^n$, (that is, $Y$ is a
local diffeomorphism) then it follows from the inverse
function theorem that: for each $i\in\{1,\ldots,n\}$, the
level surfaces $\{f_i=constant\}$ make up a codimension one
$C^2$-foliation $\F_i$ on $\R^n$. More generally, given a
$k$-combination $\{i_1,\ldots,i_k\}$ from $\{1,\ldots,n\}$,
the foliations $\F_{i_1},\ldots,\F_{i_k}$ are pairwise
transverse and the intersection
$\F_{i_1}\cap\ldots\cap\F_{i_k}$ is a $C^2$-foliation
$\F_{i_1\ldots i_k}$ of codimension $k$ on $\R^n$.

The key point in the arguments of global injectivity for maps
of $\mathbb{R}^2$ is the study of the foliations defined by
the coordinates maps. In the bidimensional case the
complexity of these foliations is given by the existence of
Reeb components (all leaves have the same topological type:
are lines). But the condition on the spectrum eliminates this
possibility.

In the higher dimensional case, that is when
$Y=(f_1,\ldots,f_n):\R^n\to\R^n$ is a $C^2$ map, the
foliations defined by the coordinates $f_i$ can be, a priori,
very complicated. In the result of Gutierrez and Maquera
mentioned above, the condition on the spectrum of $Y$
guarantees that these foliations are by planes, that is, are
homeomorphic to $\mathbb{R}^2. $ This was fundamental!

Notice that $\F_i$ has no compact leaves, for all
$i=1,\ldots,n$. In fact, otherwise the complement of a
compact leaf contains a bounded connected component $N$.
Then, either $\inf\{f_i(p);p\in\bar N\}$ or
$\sup\{f_i(p);p\in\bar N\}$ is a critical value for
$f_i:\R^n\to\R$, which contradicts the fact that this
function is a submersion. By using classical arguments of the
theory of foliations, we obtain our first result in
Proposition \ref{pro:compact} that generalizes this phenomena
to the intersected foliations $\F_{i_1\ldots i_k}$.

Another interesting property on such foliations is:
\begin{rem}\label{rem:necessary} If
$Y=(f_1,\ldots,f_n):\R^n\to\R^n$ is an injective
polynomial local diffeomorphism, then the leaves in
$\F_{i_1\ldots i_{n-2}}$ are simply connected.
\end{rem}
In fact, by a result of Bialynicki-Birula and Rosenlicht (see
Theorem \ref{teo:BR}), $Y$ is a global diffeomorphism onto
$\R^n$. Now given an ordered $(n-2)$-combination
$\{i_1,\ldots,i_{n-2}\}$ of $\{1,\ldots,n\}$, let $\{a,b\}$
be equal to $\{1,\ldots,n\}\setminus\{i_1,\ldots,i_{n-2}\}$
and note that the leaves of $\F_{i_1\ldots i_{n-2}}$ are
given by $Y^{-1}(L)$ where
$$L=\{c_{i_1},\ldots,c_{a-1},\R,c_{a+1},\ldots,c_{b-1},\R,c_{b+1},\ldots,c_{i_{n-2}}\},$$
for $c_{i_1},\ldots,c_{i_{n-2}}\in\R$. Since, in this case,
$Y^{-1}$ is a diffeomorphism we may conclude that the leaves
in $\F_{i_1\ldots i_{n-2}}$ are simply connected.

So this condition on such leaves is necessary to the
bijectivity of $Y$, our next result shows that this condition
together with $\codim(S_Y)\geq2$ is also sufficient.

Related to the Jelonek's Conjecture, in the Section
\ref{sec:poly}, we prove the following

\begin{thmA}
\label{teo:injectivityn} Let $Y=(f_1,\ldots,f_n):\R^n\to\R^n$
be a polynomial local diffeomorphism. Then, the leaves of
$\F_{i_1\ldots i_{n-2}}$ are simply connected, for all
$(n-2)$-combination $\{i_1,\ldots,i_{n-2}\}$ of
$\{1,\ldots,n\}$ and $\codim(S_Y)\geq 2$ if, and only if, $Y$
is a bijection.
\end{thmA}

This result extends the three dimensional result of Gutierrez
and Maquera \cite{GM} to the $n-$dimensional case. Note that
the Example \ref{ex:Pinchuk} shows that we can not consider
just the conditions on the foliations $\F_{i_1\ldots
i_{n-2}}$ in Theorem A.

In Section \ref{sec:sa}, we discuss about the more general
case where $Y$ is a semialgebraic mapping. We stress some
differences with the polynomial case, the most important is
that the injectivity does not implies, in general,
surjectivity. However, we obtain the following stronger
version on bijectivity:

\begin{thmB}
\label{teo:injectivitynsa} Let
$Y=(f_1,\ldots,f_n):\R^n\to\R^n$ be a $C^2$ semialgebraic
local diffeomorphism such that for all $(n-2)$-combination
$\{i_1,\ldots,i_{n-2}\}$ of $\{1,\ldots,n\}$ the leaves of
$\F_{i_1\ldots i_{n-2}}$ are simply connected. If
$\codim(S_Y)\geq 2$ then $Y$ is a bijection.
\end{thmB}


\section{Foliations and local diffeomorphisms}
\label{sec:FLC}

A fundamental notion to understand the structure of a
foliation is the concept of \emph{minimal set}, which is a
nonempty compact set given by an union of leaves and having
no proper subset satisfying these conditions. We will need
the next result due to Sacksteder \cite{S} which relates the
existence of minimal sets, distinct from compact leaves, with
not trivial holonomy for the foliation.

\begin{teo}
\label{teo:Sacksteder} Suppose that $M$ is a minimal set of a
codimension one foliation of class $C^2$ which is neither a
single compact leaf nor the entire manifold. Then for some
leaf in $M$ there is an element in its holonomy group whose
derivative (at the fixed point corresponding to the leaf) has
absolute value $< 1$. In particular, the leaf has nontrivial
holonomy and fundamental groups.
\end{teo}

Now we are ready to prove:

\begin{pro}
\label{pro:compact} Let $Y=(f_1,\ldots,f_n):\R^n\to\R^n$ be a
$C^2$ local diffeomorphism. Then the foliations
$\F_{i_1\ldots i_k}$ has no compact leaves, where
$k\in\{1,\ldots,n-1\}$ and $\{i_1,\ldots,i_k\}$ is an
arbitrary $k$-combination from $\{1,\ldots,n\}$.
\end{pro}
\begin{dem}
We will apply finite induction on $n-k$, the dimension of the
leaves. For $n-k=1$, suppose that $\ell$ is a compact leave
of $\F_{i_1\ldots i_{n-1}}$. Since $\ell$ is a compact
$1$-manifold it is homeomorphic to $S^1$. Note that $\ell$ is
contained in a dimension two leaf $L$ of $\F_{i_1\ldots
i_{n-2}}$. In fact, $Y$ be not singular implies that $\ell$
is a leaf of the regular foliation $\F_{i_{n-1}}|_L$ which is
induced by the submersion $f_{i_{n-1}}|_L:L\to\R$, in
particular we have that $\F_{i_{n-1}}|_L$ has trivial
holonomy. Hence, there exists a neighborhood $C$ of $\ell$ in
$L$ such that every leaf of $\F_{i_{n-1}}|_L$ passing through
a point of $C$ is homeomorphic to $S^1$ and is not homotopic
to a point in $L$.

Since $\F_{i_n}|_L$ is transversal to $\F_{i_{n-1}}|_L$, the
leaves of $\F_{i_{n}}|_L$ restricted to $C$ are curves
starting at one connected component of $\partial C$ and
ending at the other one.

Let $D$ be a smoothly immersed (in $\R^n$) open $2$-disc
containing $\ell$, which we may assume to be in general
position with respect to $\F_{i_n}$. Let $\G_n$ be the
foliation (with singularities) on $D$ which is induced by
$\F_{i_n}$. Then $\G_n$ is transverse to $\ell$.

Observe that $\G_n$ has no limit cycles, otherwise exists a
leaf of $\G_n$ which spirals towards a limit cycle $\gamma$
and the leaf of $\F_{i_n}$ containing $\gamma$ would have a
non trivial holonomy group, what is a contradiction.
Therefore the singularities of $\G_n$ can not be centers nor
saddles, which contradicts the fact that $D$ is in general
position with respect to $\F_{i_n}$. This contradiction
proves that $\ell$ is not compact and so the first step of
induction.

Now, by hypothesis of induction, suppose that the foliations
$\F_{i_1,\ldots,i_k}$ of dimension $r=n-k$ with $1<r<n-1$ has
no compact leaves. We will prove that the $(r+1)$-dimensional
foliations $\F_{i_1,\ldots,i_{k-1}}$ has also no compact
leaves.

Suppose, by contradiction, that there is a compact
$(r+1)$-dimensional leaf $L$ of $\F_{i_1\ldots i_{k-1}}$, for
some $k$-combination $\{i_1,\ldots,i_{k}\}$ from
$\{1,\ldots,n\}$. Since $Y$ is not singular, $\F_{i_{k}}|L$
is a foliation without holonomy which leaves are in the same
time leaves of the foliation $\F_{i_1\ldots i_{k}}$, in
particular they are $r$-dimensional. Hence, $\F_{i_{k}}|L$
has no compact leaves.

It is well known that any foliation on a compact manifold has
a minimal set. Since the leaves of $\F_{i_{k}}|L$ are not
compact, a minimal set $M$ of this foliation is $L$ or a
proper subset of $L$ which is not a single leaf. But, in the
first case by the transitivity of $M$ and in the second by
Theorem \ref{teo:Sacksteder}, we may conclude that the
holonomy of $\F_{i_{k}}|L$ is not trivial. This contradiction
proves that $L$ can not be compact.
\end{dem}


\section{Global injectivity: polynomial case}
\label{sec:poly}

In this section we will prove Theorem A which generalizes to
$\R^n$ the bijectivity result obtained in \cite{GM} to
polynomial local diffeomorphisms of $\R^3$ into itself. To do
this we shall need the following results due to Jelonek
\cite{J}:

\begin{teo}
\label{teo:J1} If $Y:\R^n\to\R^n$ is a polynomial local
diffeomorphism and $\codim(S_Y)\geq3$, then $Y$ is a
bijection (and consequently $S_Y=\vazio$).
\end{teo}

\begin{teo}
\label{teo:J2} Let $Y:\R^n\to\R^n$ be a non-constant
polynomial mapping. Then the set $S_Y$ is closed,
semialgebraic and for every non-empty connected component
$S\subset S_Y$ we have $1\leq\dim(S)\leq n-1$. Moreover, for
every point $q\in S_Y$ there exists a polynomial mapping
$\phi :\R\to S_Y$ such that $\phi (\R)$ is a semi-algebraic
curve passing through $q$.
\end{teo}

The following result due to Bialynicki-Birula and Rosenlicht
\cite{BR} say us that to obtain bijectivity of polynomial
mappings is enough to take its injectivity.

\begin{teo}
\label{teo:BR} If $Y:\R^n\to\R^n$ is an injective polynomial
mapping, then $Y$ is also surjective.
\end{teo}

In the following we use the standard notation $\pi_1(\el)$ to
the fundamental group of a topological space $\el$. The proof
of the following lemma is easy and will be omitted.

\begin{lem}
\label{lem:isomorphism} Let $Y:\R^n\to\R^n$ be a $C^1$-map
such that $\spec(Y)\cap\{0\}=\vazio$ and $\pi_1(\el)$ trivial
for all $\el\in\F_{i_1\ldots i_{n-2}}$, where $\{i_1,\ldots,
i_{n-2}\}$ is a $(n-2)$-combination of $\{1,\ldots,n\}$. Let
$A:\R^n\to\R^n$ be a linear isomorphism. If $Z=A\circ Y\circ
A^{-1}$ then $\spec(Y)$=$\spec(Z)$, $S_Z=A(S_Y)$ and
$\pi_1(\el_Z)$ is trivial, for all $\el_Z\in\F^Z_{i_1\ldots
i_{n-2}}$, where $\F^Z$ indicates foliations induced by
coordinate functions of $Z$.
\end{lem}

The next result proved in \cite{DT} by Dru\.{z}kowski and
Tutaj provides an estimate of the cardinality of preimages of
any point to a polynomial local diffeomorphism.

\begin{teo}
\label{teo:finite} Let $Y:\R^n\to\R^n$ be a polynomial local
diffeomorphism. Then, given $q\in\R^n$, the equation $Y(p)=q$
has a finite number of solutions and
$$\#\{p\in\R^n;\ Y(p)=q\}\leq \overset{n}{\underset{i=1}{\Pi}}\dg(f_i).$$
\end{teo}

The next result ``is contained" in Theorem A but it was
remained here because its proof provides a helpful
geometrical viewpoint of the reasoning also applied in the
last one.

\begin{teo}
\label{teo:injectivity4} Let $Y=(f_1,\ldots,f_4):\R^4\to\R^4$
be a polynomial map such that $\spec(Y)\cap\{0\}=\vazio$ and
$\pi_1(\el)$ is trivial, for all $\el\in\F_{ij}$, and all
$i,j\in\{1,\ldots,4\}$ with $i\neq j$. If cod$(S_Y)\geq 2$
then $Y$ is a bijection.
\end{teo}
\begin{dem}
Suppose that $Y$ is not bijective. By Theorem \ref{teo:J1},
we must have cod$(S_Y)=$dim$(S_Y)=2$. Since $Y$ is a local
diffeomorphism $Y(\R^4)$ is open and applying Theorem
\ref{teo:J2} we obtain that
\begin{enumerate}
    \item[(a)] $Y(\R^4)\supset\R^4\setminus S_Y$.
\end{enumerate}

Analogously to Gutierrez and Maquera in \cite{GM} we will
construct a compact neighborhood $W$ in $\R^4$ of a not
proper point $p\in S_Y$ that has compact pre-image, which is
a contradiction.

By Theorem \ref{teo:J2} and Lemma \ref{lem:isomorphism}, we
may suppose that $S_Y$ contains an analytic regular curve
$\gamma:(a_1-\delta_1,a_1+\delta_1)\to\R^4$ meeting the
hyperplane $\{x_1=a_1\}$ transversally at the point
$p=\gamma(a_1)=(a_1,\ldots,a_4)$. Furthermore, for a
sufficiently small $\delta_2>0$, we can suppose that the
$3$-disc $D(a_1)=\{a_1\}\times D \subset \{x_1=a_1\}$, where
$D$ is the 3-disc of $\R^3$ centered at $(a_2,a_3,a_4)$ with
radius $\delta_2$, satisfies:
\begin{enumerate}
    \item[(b)] $\gamma(a_1)=D(a_1)\cap\gamma$, $C(a_1)\cap\gamma=\vazio$,
    where $C(a_1)$ is the boundary of $D(a_1)$, and $D(a_1)\cap S_Y$ has
    an injective projection at the $x_2$-axis. That is, exist
    $\eps_1,\eps_2>0$ such that
    $$\Pi_2|_{D(a_1)\cap S_Y}:D(a_1)\cap S_Y\to[a_2-\eps_1,a_2+\eps_2]$$
    is bijective and so an homeomorphism with inverse
    $$\varphi:[a_2-\eps_1,a_2+\eps_2]\to D(a_1)\cap S_Y,$$
    that associates each $t\in[a_2-\eps_1,a_2+\eps_2]$ to
    $\varphi(t)=(\Pi_2|_{D(a_1)\cap S_Y})^{-1}(t)=\Pi_2^{-1}(t)\cap D(a_1)\cap
    S_Y$ (see Figure \ref{fig:dem2}).
\end{enumerate}

In this way, to each $t\in(a_2-\eps_1,a_2+\eps_2)$ we can
associate an unique 2-disc $D(t)\subset D(a_1)\cap \{x_2=t\}$
which satisfies $D(t)\cap S_Y=\{\varphi(t)\}$ and $C(t)\cap
S_Y=\vazio$, where $C(t)$ is the boundary of $D(t)$.
\begin{enumerate}
    \item[(c)] It is well known that $\# Y^{-1}$ is locally
    constant at proper points of $Y$ and by Theorem \label{teo:finita} there exists a
    positive integer $K$ such that for all $q\in\R^4$, $\#Y^{-1}(q)\leq
    K$.
\end{enumerate}
\begin{figure}[h]
\begin{center}
\includegraphics[width=9cm]{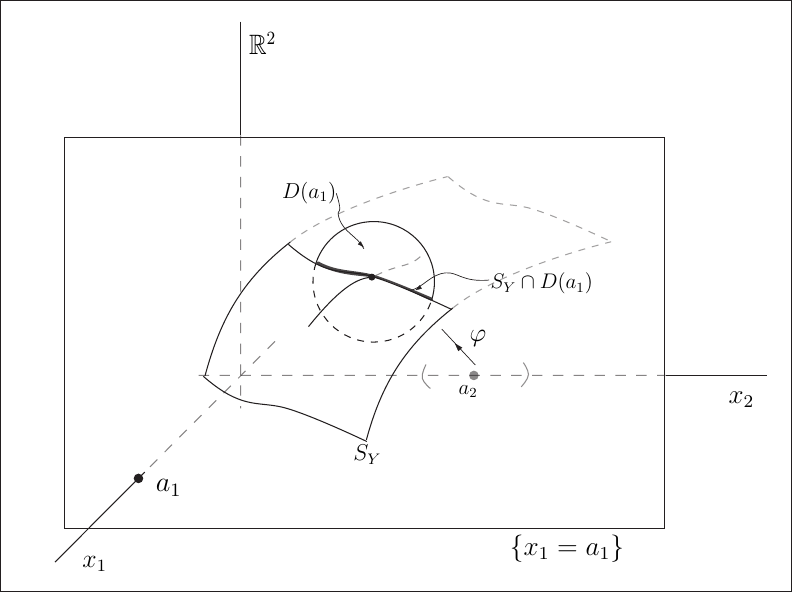}
\caption{\label{fig:dem2} }
\end{center}
\end{figure}
These imply that $Y^{-1}(C(t))$ is an union of finitely many
embedded circles $C_1(t),\ldots,C_k(t)$ contained in
$f_1^{-1}(a_1)\cap f_2^{-1}(t)$, that is, contained in leaves
of $\F_{12}$. Note that each $Y|_{C_j(t)}:C_j(t)\to C(t)$ is
a finite covering. As, by hypothesis each connected component
of $f_1^{-1}(a_1)\cap f_2^{-1}(t)$ is a simply connected
2-submanifold of $\R^4$ and by Proposition \ref{pro:compact}
it is not compact, we have that each of this leaves is a
plane. Hence, for each $j=1,2,\ldots,k$ there exists a
compact 2-disc $D_j(t)\subset f_1^{-1}(a_1)\cap f_2^{-1}(t)$
bounded by $C_j(t)$. It follows that $Y(D_j(t))=D(t)$, for
all $j\in\{1,\ldots,k\}$. As $D(t)$ is simply connected
$Y|_{D_j(t)}:D_j(t)\to D(t)$ is a bijective map and so a
diffeomorphism, for all $j\in\{1,\ldots,k\}$.

As $D(t)\cap S_Y=\varphi(t)$ and $\# Y^{-1}$ is locally
constant at proper points, $\# Y^{-1}$ must be identically
equal to $k$ in $D(t)\setminus\{\varphi(t)\}$ and so
$Y^{-1}(D(t)\setminus\{\varphi(t)\})\subset\cup_{j=1}^{k}D_j(t)$.
For the other hand, $\#Y^{-1}(\varphi(t))\geq k$, but since
$Y$ is a local diffeomorphism, we conclude
\begin{enumerate}
    \item [(d)] for all $q\in D(t)$, $\#Y^{-1}(q)=k$ and
    $Y^{-1}(D(t))=\cup_{j=1}^{k}D_j(t)$.
\end{enumerate}

Taking $\overline{\delta}_2=$min$\{\eps_1,\eps_2\}$, we have
just seen that for all
$t\in[a_2-\overline{\delta}_2,a_2+\overline{\delta}_2]=I_2$,
$Y^{-1}(D(t))=\cup_{j=1}^{k}D_j(t)$ (for the same $k$), where
$D(t)$ and $D_j(t)$ are as previously defined. Hence, we may
define $W(a_1)=\cup_{t\in I_2}(D(t))$ which is a 3-compact se
containing $\gamma(a_1)$, contained at $D(a_1)$ and
satisfying $Y^{-1}(W(a_1))$ is an union of $k$ 3-compact sets
in $\R^4$.

Now, notice that reducing $\delta_1$, if necessary, we have
that

\begin{enumerate}
    \item [(e)] for all
    $s\in[a_1-\delta_1,a_1+\delta_1]=I_1$, the 3-disc $D(s)=\{s\}\times
    D$ and its boundary $C(s)=\{s\}\times\partial D$, where
    $D$ is as previously defined, satisfy
    $\{\gamma(s)\}=D(s)\cap\gamma$ and
    $C(s)\cap\gamma=\vazio$.
\end{enumerate}

Proceeding as above, we may construct a 3-compact set $W(s)$,
for each $s\in I_1$, contained in $D(s)$ such that
$Y^{-1}(W(s))$ is an union of $k$ 3-compact sets in $\R^4$.
In fact, given an $s\in I_1$ we can choose continuously a
$\overline{\delta}_2(s)$ in such way that $W(s)=\cup_{t\in
I_2(s)}D(s,t)$, where
$I_2(s)=[a_2-\overline{\delta}_2(s),a_2+\overline{\delta}_2(s)]$,
contains $\gamma(s)$ and satisfies all the required
properties in the beginning of this paragraph.

Since $\overline{\delta}_2:I_1\to\R_+$ is a continuous
function defined in a compact set and not achieve 0, we have
that $\delta=\inf_{s\in I_1}\{\overline{\delta}_2(s)\}$ is
positive. In this way, for all $s\in I_1$ the set
$W(s)=\cup_{t\in[a_2-\delta,a_2+\delta]}D(t)$ has as
pre-image of Y an union of $k$ 3-compact sets in $\R^4$.

Therefore, since $W=\cup_{s\in I_1}W(s)$ is a compact
neighborhood of $p=\gamma(a_1)$ in $\R^4$ such that
$Y^{-1}(W)$ is a compact set, we obtain a contradiction with
the assumption $p\in S_Y$.
\end{dem}


The proof of the next result is analogous to the previous one
so we give a shorter version emphasizing the differences
between $n=4$ and $n>4$ cases.

\begin{thmA}
\label{teo:injectivityn} Let $Y=(f_1,\ldots,f_n):\R^n\to\R^n$
be a polynomial map such that $\spec(Y)\cap\{0\}=\vazio$ and
for all $(n-2)$-combination $\{i_1,\ldots,i_{n-2}\}$ of
\{1,\ldots,n\} the leaves of $\F_{i_1\ldots i_{n-2}}$ are
simply connected. If cod$(S_Y)\geq 2$ then $Y$ is a
bijection.
\end{thmA}
\begin{dem}
Since we enunciate the result for all $n\geq1$, we start
stressing that for $n=1$ the result is trivial, the case
$n=2$ was proved by Jelonek in \cite{J} just with the
codimension hypothesis and finally the case $n=3$ is a
consequence of a result in \cite{GM} by Gutierrez and
Maquera.

To deal with $n\geq 4$ case, suppose that $Y$ is not
bijective. By Theorem \ref{teo:J1}, we must have
cod$(S_Y)=2$. Since $Y$ is a local diffeomorphism $Y(\R^n)$
is open and applying Theorem \ref{teo:J2} we obtain that
\begin{enumerate}
    \item[(a)] $Y(\R^n)\supset\R^n\setminus S_Y$.
\end{enumerate}

Here again the idea is construct a compact neighborhood in
$\R^n$ of a not proper point $p\in S_Y$ that has compact
pre-image, which is a contradiction.

By Theorem \ref{teo:J2} and Lemma \ref{lem:isomorphism}, we
may suppose that $S_Y$ contains an analytic regular curve
$\gamma:(a_1-\delta_1,a_1+\delta_1)\to\R^n$ meeting the
hyperplane $\{x_1=a_1\}$ transversally at the point
$p=\gamma(a_1)=(a_1,\ldots,a_n)$. Furthermore, for a
sufficiently small $\delta>0$, we can suppose that the
$(n-1)$-disc $D(a_1)=\{a_1\}\times D \subset \{x_1=a_1\}$,
where $D$ is the $(n-1)$-disc of $\R^{n-1}$ centered at
$(a_2,\ldots,a_n)$ with radius $\delta$, satisfies:
\begin{enumerate}
    \item[(b)] $\gamma(a_1)=D(a_1)\cap\gamma$, $C(a_1)\cap\gamma=\vazio$,
    where $C(a_1)$ is the boundary of $D(a_1)$, and $\Gamma=D(a_1)\cap S_Y\cap\{x_2=t_2\}\cap\ldots\cap\{x_{n-3}=t_{n-3}\}$,
     where $t_j\in(a_j-\delta,a_j+\delta)$ is fixed for each $j\in\{2,\ldots,n-3\}$, has
    injective projection at the $x_{n-2}$-axis. That is, exist
    $\eps_1,\eps_2>0$ such that
    $$\Pi_{n-2}|_{\Gamma}:\Gamma\to[a_{n-2}-\eps_1,a_{n-2}+\eps_2]$$
    is bijective and so is an homeomorphism with inverse
    $$\varphi_{n-2}:[a_{n-2}-\eps_1,a_{n-2}+\eps_2]\to \Gamma,$$
    that associates each $t_{n-2}\in[a_{n-2}-\eps_1,a_{n-2}+\eps_2]$ to
    $\varphi_{n-2}(t_{n-2})=(\Pi_{n-2}|_{\Gamma})^{-1}(t_{n-2})=\Pi_{n-2}^{-1}(t_{n-2})\cap \Gamma$.
\end{enumerate}

In this way, to each
$t_{n-2}\in(a_{n-2}-\eps_1,a_{n-2}+\eps_2)$ we can associate
an unique 2-disc $D(a_1,t_2,\ldots,t_{n-2})\subset D(a_1)\cap
\{x_2=t_2\}\cap\ldots\cap\{x_{n-3}=t_{n-3}\}\cap\{x_{n-2}=t_{n-2}\}$
which satisfies $D(a_1,t_2,\ldots,t_{n-2})\cap
S_Y=\{\varphi_{n-2}(t_{n-2})\}$ and
$C(a_1,t_2,\ldots,t_{n-2})\cap S_Y=\vazio$, where
$C(a_1,t_2,\ldots,t_{n-2})$ is the boundary of
$D(a_1,t_2,\ldots,t_{n-2})$.

Since $C(a_1,t_2,\ldots,t_{n-2})$ is a circle contained in
$\{x_1=a_1\}\cap\{x_2=t_2\}\cap\ldots\cap\{x_{n-2}=t_{n-2}\}$,
we have that $Y^{-1}(C(a_1,t_2,\ldots,t_{n-2}))\subset
f_1^{-1}(a_1)\cap f_2^{-1}(t_2)\cap\ldots\cap
f_{n-2}^{-1}(t_{n-2})$, that means
$Y^{-1}(C(a_1,t_2,\ldots,t_{n-2}))$ is contained in leaves of
$\F_{1,2,\ldots,n-2}$. By hypothesis such leaves are simply
connected and by Proposition \ref{pro:compact} they are not
compact and so homeomorphic to planes. Therefore following
the reasoning of \cite{GM}, or equivalently, of Theorem
\ref{teo:injectivity4}, we have that
\begin{enumerate}
    \item[(c)] $Y^{-1}(C(a_1,t_2,\ldots,t_{n-2}))$ is the
    union of finitely many embedded circles
    $$C_1(a_1,t_2,\ldots,t_{n-2}),\ldots,C_k(a_1,t_2,\ldots,t_{n-2})$$
    and each of them bounds a 2-disc
    $$D_1(a_1,t_2,\ldots,t_{n-2}),\ldots,D_k(a_1,t_2,\ldots,t_{n-2})$$
    contained on $f_1^{-1}(a_1)\cap f_2^{-1}(t_2)\cap\ldots\cap
f_{n-2}^{-1}(t_{n-2})$ satisfying:
$$Y^{-1}(D(a_1,t_2,\ldots,t_{n-2}))=\displaystyle\cup_{j=1}^{k}D_j(a_1,t_2,\ldots,t_{n-2}).$$
\end{enumerate}

Proceeding as in Theorem \ref{teo:injectivity4}, we may
choose $\delta_{n-2}<$min$\{\eps_1,\eps_2\}$ such that for
all
$t_{n-2}\in[a_{n-2}-\delta_{n-2},a_{n-2}+\delta_{n-2}]=I_{n-2}$,
we have
$Y^{-1}(D(a_1,t_2,\ldots,t_{n-2}))=\cup_{j=1}^kD_j(a_1,t_2,\ldots,t_{n-2})$
for the same $k$, once that the number of pre-image of $Y$ is
locally constant at proper points.

\begin{enumerate}
    \item [(d)] Therefore the set $W_{n-2}(a_1,t_2,\ldots,t_{n-3})=\cup_{t_{n-2\in
    I_{n-2}}}D(a_1,t_2,\ldots,t_{n-2})$ is a 3-compact set
    that contains $p\in S_Y$ and
    $Y^{-1}(W_{n-2}(a_1,t_2,\ldots,t_{n-3}))$ is an union of
    $k$ 3-compact sets of $\R^n$.
\end{enumerate}

In the same way,
\begin{enumerate}
    \item [(e)] we may choose $\delta_{n-3}>0$ such that for all
$t_{n-3}\in[a_{n-3}-\delta_{n-3},a_{n-3}+\delta_{n-3}]=I_{n-3}$,
the set $W_{n-2}(a_1,t_2,\ldots,t_{n-3})$ as defined in item
(d) satisfies all properties described in this item. Hence,
we can define
$$W_{n-3}(a_1,t_2,\ldots,t_{n-4})=\cup_{t_{n-3}\in
I_{n-3}}W_{n-2}(a_1,t_2,\ldots,t_{n-3})$$
which is a
4-compact set containing $p\in S_Y$ and
$Y^{-1}(W_{n-3}(a_1,t_2,\ldots,t_{n-4}))$ is an union of $k$
4-compact sets of $\R^n$.
\end{enumerate}

After $n-3$ steps as (e), we have constructed the set
$W_2(a_1)=\cup_{t_2\in I_2}W_3(a_1,t_2)$ which is an
$(n-1)$-compact set containing $p\in S_Y$ and such that
$Y^{-1}(W_2(a_1))$ is an union of $k$ $(n-1)$-compact sets of
$\R^n$.

Finally, for $\delta_1>0$ small enough, we may repeat the
reasoning at each step above to obtain for each
$t_1\in[a_1-\delta_1,a_1+\delta_1]=I_1$ the set $W_2(t_1)$,
which is an $(n-1)$-compact containing $p$ and with pre-image
of $Y$ equal the union of $k$ $(n-1)$-compact sets of $\R^n$.
Therefore, $W_1=\cup_{t_1\in I_1} W_2(t_1)$ is the required
compact neighborhood in $\R^n$ of $p\in S_Y$, since
$Y^{-1}(W_1)$ is an union of $k$ $n$-compact sets and it is
compact.
\end{dem}


\section{Semialgebraic mappings}
\label{sec:sa}

In this section we establish analogous results on global
injectivity and bijectivity to \emph{semialgebraic mappings}
$Y:\R^n\to\R^n$, which are mappings with semialgebraic
graphs. In the following we describe some facts about
semialgebraic sets and mappings.

Every semialgebraic set $S\subset\R^n$ allow finite
stratification, that means, can be decomposed as a disjoint
finite union $S=\cup_{i=1}^d S_i$ where each $S_i$ is an
analytical submanifold of $\R^n$. In particular, every
semialgebraic discrete subset $A\subset\R^n$ is finite. If
$Y=(f_1,\ldots,f_n):\R^n\to\R^n$ is semialgebraic then
$Y^{-1}(S)$ is a semialgebraic set for all $S\subset\R^n$
semialgebraic.

The next result can be found at \cite{C} and allows us to
obtain the finiteness of the fibers as in Theorem
\ref{teo:finite}.

\begin{teo}
\label{teo:hardt}[Hardt] Let $A\subset\R^n$ be a
semialgebraic set and $Y:A\to\R^m$, a continuous
semialgebraic mapping. There is a finite semialgebraic
partition of $\R^m$ into $S_1,\ldots,S_d$ such that the
restriction of $Y$ to each $Y^{-1}(S_i)$ is a globally
trivial bundle. In particular, given $p,q\in S_i$,
$Y^{-1}(p)$ is homeomorphic to $Y^{-1}(q)$.
\end{teo}

\begin{cor}
\label{cor:finitesa} If $Y:\R^n\to\R^n$ is a semialgebraic
local homeomorphism, then there exists a $k\in\N$ such that
$\# Y^{-1}(p)\leq k$, for all $q\in\R^n$.
\end{cor}
\begin{dem}
Since $Y$ is a semialgebraic local homeomorphism, the
semialgebraic subsets $Y^{-1}(p)$ with $p\in\R^n$ are
discrete, in particular, finite. By Hardt's Theorem, there
exists a finite partition of $\R^n$ in semialgebraic subsets
$S_1,\ldots,S_k$ in a such way that given $p,q\in S_i$,
$Y^{-1}(p)$ is homeomorphic to $Y^{-1}(q)$, in particular
$Y^{-1}(p)$ and $Y^{-1}(q)$ have the same cardinality and
this prove the corollary.
\end{dem}

Another consequence of Hardt's Theorem is:
\begin{cor}
\label{cor:localconstante} Let $Y:\R^n\to\R^n$ be a
semialgebraic local diffeomorphism. Then $Y$ is proper at
$p_0\in\R^n$ if, and only if, the function $h:\R^n\to\N$ that
associated each $p\in\R^n$ to $h(p)=\# X^{-1}(p)$ is locally
constant at $p_0$.
\end{cor}

To prove that in the case of semialgebraic mappings the set
of not proper points $S_Y$ is also semialgebraic, we shall
need the following result given, for example, in \cite{C}:

\begin{teo}\label{teo:TS}[Tarski-Seidenberg]
Let $A$ be a semialgebraic subset of $\R^n\times\R^m$ and
$\pi:\R^n\times\R^m\to\R^m$ the projection $\pi(p,q)=q$. Then
$\pi(A)$ is a semialgebraic subset of $\R^m$.
\end{teo}

\begin{pro}
\label{prop:sa} If $Y:\R^n\to\R^n$ is a semialgebraic
mapping, then $S_Y$ is semialgebraic.
\end{pro}
\begin{dem}
Note that $p\in S_Y$ if, and only if, for each $\eps>0$ there
exists $q\in\R^n$ such that $|Y(q)-p|<\eps$ and $|q|>1/\eps$.
Now, Let $\Sigma_Y$ denote the subset of
$\R^n\times\R^n\times(0,+\infty)$ defined by:
$$\Sigma_Y=\{(q,p,\eps)\in\R^n\times\R^n\times(0,+\infty);\
|q|>1/\eps,\ \text{and}\ |Y(q)-p|<\eps\}.$$

we claim that $\Sigma_Y$ is semialgebraic subset of
$\R^n\times\R^n\times(0,+\infty)$. This follows of the basic
fact that if $A$ is semialgebraic and $f:A\to\R$ is a
semialgebraic function, then the set $\{q\in A;\ f(p)>0\}$ is
semialgebraic (cf. \cite{C}). Since $S_Y=\pi(\Sigma_Y)$,
where $\pi$ is the projection
$$\pi:\R^n\times\R^n\times(0,+\infty)\to\R^n$$
defined by $\pi(x,y,\eps)=y$. The result follows directly
from the Tarski-Seidenberg's Theorem.
\end{dem}

In Lemma 8.1 of \cite{J}, Jelonek proved:

\begin{lem}\label{lem:connected} Let $S$ be a closed semialgebraic subset of $\R^n$. If $\codim(S)\geq2$,
then the set $\R^n\setminus S$ is connected. If $codim(S)\geq3$, then the set $\R^n\setminus S$ is
simply connected.
\end{lem}

Now, let $S_1,\ldots,S_d$ be a family of analytical
submanifolds of $\R^n$ all of codimension $r$ with $1\leq
r\leq n-1$, $p\in S_i$, for some $i\in\{1,\ldots,d\}$, and
$G$ the Grassmannian of $r-$planes in $\R^n$ with origin at
$p$. By Thom's Transversality Theorem the subset of $G$
defined by $T=\{L\in G;\ L \text{ is transversal to each }
S_i\}$ is a residual subset. In particular, the set $T$ is
not empty. This establishes the following result:

\begin{pro}
\label{pro:trans} Let $S_1,\ldots,S_d$ be a family of
analytical submanifolds of $\R^n$ all of codimension $1\leq
r\leq n-1$ and $p\in S_i$, for some $i\in\{1,\ldots,d\}$.
Then there exists a $r-$dimensional affine plane passing
trough $p$ intersecting transversally each $S_j$. In
particular, the intersection of a plane $L\in T$ with $S_j$
is a discrete set to each $S_j$.
\end{pro}

On global injectivity results, we start observing that the Theorem \ref{teo:J1} is partially
true to semialgebraic maps. In fact, following its proof in
\cite{J} we obtain the global injectivity result:

\begin{teo}
\label{teo:J1sa} If $Y:\R^n\to\R^n$ is a semialgebraic local
homeomorphism and $\codim(S_Y)\geq3$, then $Y$ is injective.
\end{teo}

The only barrier to take bijection in this result is that
to semialgebraic mappings injectivity does not imply
surjectivity which is given, in the polynomial case, by
Theorem \ref{teo:BR}.

For the other hand, Kurdyka and Rusek in \cite{KR} studied
and presented a family of injective semialgebraic mapping
that are surjective, which includes all mappings with
algebraic graphs. They characterize this family using
properties of so-called arcwise symmetric sets. For the sake
of completeness we give this definition and some notions
necessary to enunciate their result. For more details see
\cite{KR}. A subset $E$ of a real analytic manifold $M$ is
said to be \emph{arcwise symmetric} if, for every analytic
mapping $\gamma:(-1,1)\to M$ the interior of the set
$\gamma^{-1}(E)$ is either empty, or covers the interval
$(-1,1)$. Suppose now that $X$ is a compact algebraic variety
and $i:\R^n\to X$ is a biregular isomorphism of $\R^n$ onto
$i(\R^n)$ which is Zariski dense and open in $X$. Then the
pair $(X,i)$ is called an \emph{algebraic compactification}
of $\R^n$. We are ready to see Theorem 4.1 of \cite{KR}:

\begin{teo}
\label{teo:KR} A continuous injective, semialgebraic mapping
$Y:\R^n\to\R^n$ is surjective if there exist an algebraic
compactification $(X,i)$ of $\R^n\times\R^n$ and $E$, arcwise
symmetric in $X$, such that $i(\R^n,Y(\R^n))=E\cap
i(\R^n\times\R^n)$.
\end{teo}

A class of continuous injective mappings
$Y:\R^n\to\R^n$ satisfying the additional condition of this theorem is
those with algebraic graphs. Furthermore, they gave the
example $Y(p)=\sqrt{p^2+1}-p$ from $\R$ to itself that is a
injective semialgebraic local diffeomorphism that is not
surjective.

Therefore with the additional condition just below discussed
we may obtain bijectivity on Theorem \ref{teo:J1sa}. In the
following we provide another such a sufficient condition.
This result also generalizes to semialgebraic mappings the
Theorem A.

\begin{thmB}
\label{teo:injectivitynsa} Let
$Y=(f_1,\ldots,f_n):\R^n\to\R^n$ be a $C²$ semialgebraic local
diffeomorphism such that for all $(n-2)$-combination
$\{i_1,\ldots,i_{n-2}\}$ of $\{1,\ldots,n\}$ the leaves of
$\F_{i_1\ldots i_{n-2}}$ are simply connected. If
$\codim(S_Y)\geq 2$ then $Y$ is a bijection.
\end{thmB}
\begin{dem}
We will prove that the set of not proper points $S_Y$ is
empty and $Y$ is a surjective map. So, by Hadamard's Theorem $Y$ is a global
diffeomorphism. Suppose, by contradiction, that
$S_Y\neq\vazio$ and that the $\codim(S_Y)=r\geq2$. We have
that $\R^n\setminus S_Y$ is connected and so by Corollary
\ref{cor:finitesa} there exists $k\in\N$ such that
$\# Y^{-1}(q)=k$, for all $q\in\R^n\setminus S_Y$. Given
$p\in S_Y$, by Lemma \ref{lem:isomorphism} and Proposition
\ref{pro:trans} we may assume that the $r-$dimensional plane
$L=\{(p_1,\ldots,p_{n-r},x_{n-r+1},\ldots,x_n) ;
x_{n-r+1},\ldots,x_{n}\in\R\}$ is transversal to all strata
of a stratification of $S_Y$. In particular, $S_Y$ intersects
the plane $L$ in a finite number of points. Hence, there
exists a circle $C$ in the $2-$dimensional plane
$\ell=\{(p_1,\ldots,p_{n-2},x_{n-1},x_n) ;
x_{n-1},x_{n}\in\R\}$ centered at $p$ such that $C\cap S_Y=
\vazio$. Since $C\cap S_Y=\vazio$ we have that $F^{-1}(C)$ is
the union of finitely many embedded circles $C_1,\ldots,C_k$.
Each $C_i$ is contained in a leaf of the foliation
$\F_{1\ldots n-2}$ (which is, by hypothesis and Proposition
\ref{pro:compact}, a foliation by planes), in particular
each $C_i$ is the boundary of a $2-$disc $D_i$ in the
respective leaf containing $C_i$. As $Y(D_i)$ covers the disc
in $\ell$ bounded by $C$, there is (a unique) $q_i\in D_i$
such that $Y(q_i) = p$. So
$Y^{-1}(p)\supset\{q_1,\ldots,q_k\}$ and $\# Y^{-1}(p)\geq
k$. Since $Y$ is a local diffeomorphism
$Y^{-1}(p)=\{q_1,\ldots,q_k\}$, otherwise elements
in $\R^n\setminus S_Y$ sufficiently near to $p$ would have more
than $k$ points at its pre-image. Therefore, $\# Y^{-1}(p)=k$
and, since $p$ is an arbitrary point of $S_Y$, we can
conclude that the function $\# Y^{-1}(.)$ is locally constant
at $p\in S_Y$, which is a contradiction with Corollary
\ref{cor:localconstante}. This contradiction proves that
$S_Y=\vazio$ and so the theorem.
\end{dem}

Gutierrez and Maquera proved that if $Y:\R^3\to\R^3$ is
a polynomial map such $\spec(Y)\cap[0,\eps)=\vazio$, for some
$\eps>0$, and $\codim(S_Y)\geq2$, then $Y$ is bijective (see \cite{GM}, Theorem 1.3). The following corollaries are somewhat slight generalizations of that result for semialgebraic (not necessarily polynomial) maps $Y:\R^3\rightarrow\R^3.$

\begin{cor}Let $Y:\R^3\to\R^3$ be a $C^2$ semialgebraic (not necessarily polynomial) map such that $\spec(Y)\cap(-\eps,\eps)=\vazio$, for some
$\eps>0$. If  $\codim(S_Y)\geq2$, then $Y$ is bijective.
\end{cor}

\begin{proof} Let $Y=(f_1,f_2,f_3):\R^3\to\R^3$ be a $C^2$ semialgebraic (not necessarily polynomial) map such that $\spec(Y)\cap(-\eps,\eps)=\vazio$, for some
$\eps>0$. It comes from Theorem 1.1 of \cite{GM} that the leaves of $\F_{i}$, $i=1,2,3$ are simply connected. So, by using Theorem B we get $Y$ is bijective.
\end{proof}

If the assumptions of Corollary 4.10 are relaxed to the existence of $\eps >0$ such
that $\spec (Y)$ does not intersect $[0,\eps)$, we cannot hope that the semialgebraic map $Y:\R^3\to\R^3$ should be  a surjective map. In fact, let $Y:\R^3\to\R^3$ be defined in the following way:
$$Y(x,y,z)=(\sqrt{1+x^2}-x,\sqrt{1+y^2}-y, z(1+x^2)(1+y^2)).$$ We have that $Y$ is a smooth semialgebraic map,  $\codim(S_Y)=2$, $\spec(Y)\cap [0,1)=\vazio$ and $Y$ is not surjective.

\begin{cor}Let $Y:\R^3\to\R^3$ be a $C^2$ semialgebraic (not necessarily polynomial) map such that $\spec(Y)\cap[0,\eps)=\vazio$, for some
$\eps>0$. If  $\codim(S_Y)\geq2$, then $Y$ is injective.
\end{cor}

\begin{proof} Let $Y:\R^3\to\R^3$ be a $C^2$ semialgebraic (not necessarily polynomial) map such that $\spec(Y)\cap(-\eps,\eps)=\vazio$, for some
$\eps>0$. Let us denote $Y_t(p)=Y(p)+tp$. Given $0<t<\eps$, we have that $\spec(Y_t)\cap(-a,a)=\vazio$ where $0<a<\min\{t,\eps-t\}$. It comes from above corollary that $Y_t$ is bijective for all $0<t<\eps$, hence $Y$ is injective.
\end{proof}

\section{Final comments}

\vspace{.3cm} \noindent
 {\bf I. The (complex) Jacobian Conjecture}: To prove that his conjecture implies
 the Jacobian conjecture Jelonek associated in a standard way a map
$Y_{\C}:\C^n\to\C^n$ to a $Y_{\R}:\R^{2n}\to\R^{2n}$ and
noted that when the complex codimension of $S_Y$ is 1 the
real is 2. Therefore our result establishes a new
characterization to the Jacobian Conjecture. In fact, with
our result to prove the (complex) Jacobian Conjecture it is
enough to check if the leaves of $\F_{i_1\ldots i_{2n-2}}$,
related with $Y_\R$, are simply connected.

 \vspace{.3cm}
 \noindent
 {\bf II. } Observe that a polynomial mapping $Y:\C^n\to\C^n$ has
Jacobian determinant nonzero everywhere if, and only if, it
is constant nonzero. So the Jacobian Conjecture can be
formulated in the real context as follows:

\textbf{Real Jacobian Conjecture:} If $Y:\R^n\to\R^n$ is a
polynomial mapping with constant nonzero Jacobian determinant
then $Y$ is bijective.

For the other hand, consider the following reduction result
given by Bass, Connell and Wright in \cite{bcw}:

\begin{teo}\label{teo:reduction}
If the (real or complex) Jacobian Conjecture holds for all
$n\geq2$ and all polynomial mappings of the form $I+H$, where
$I$ is the identity and $H$ is cubic homogeneous, then the
(real or complex) Jacobian Conjecture holds.
\end{teo}

Furthermore, Hubbers classified in \cite{H} all polynomial
local diffeomorphisms of the form $I+H$ above in dimension
four. His classification provided eight such mappings up to
linear conjugations. We were surprised to see that our
additional hypothesis on the foliations $\F_{i_1i_2}$ was satisfied for all this eight
mappings.

Moreover, Dru\.{z}kowski in \cite{D} sharped Theorem
\ref{teo:reduction} proving that it suffices to prove the
Jacobian Conjecture for a special class of cubic homogeneous
mappings: the cubic-linear mappings.

\begin{teo}\label{teo:reductionlinear}
It suffices to prove the (real or complex) Jacobian
Conjecture for all $n\geq2$ and all polynomial mappings of
the form $Y=(p_1+l_1^3,\ldots,p_n+l_n^3)$, where
$l_j=a_{1j}p_1+\cdots+a_{nj}p_n$, for all $j$.
\end{teo}

In \cite{H} Hubbers established the injectivity to
cubic-linear maps until dimension $n=7$. Hence, by Remark
\ref{rem:necessary}, we may conclude that to these cases our
additional condition on the coordinate foliations is also
satisfied.

So we can ask if such condition on these foliations holds for
cubic-linear maps on $\R^n$ with $n>7$.

\end{document}